\documentclass[12pt]{amsart}
\usepackage{geometry}       
\usepackage{graphicx}
\usepackage{amsmath,amscd,amssymb,dsfont,stmaryrd}
\usepackage{amssymb}
\usepackage{dsfont}
\usepackage{epstopdf}
\usepackage{mathrsfs}
\usepackage{hyperref}
\usepackage{enumerate}
\usepackage{graphicx}
\usepackage{float}
\usepackage{tikz}
\usetikzlibrary{matrix,arrows}
\newcommand{\bb}[1]{\mathbb{#1}}

\newcommand{\cal}[1]{\mathcal{#1}}

\newcommand{\raw}{\rightarrow}

\newcommand{\inv}{^{-1}}
\DeclareMathOperator{\Hom}{Hom}
\DeclareMathOperator{\GL}{GL}
\DeclareMathOperator{\SL}{SL}
\DeclareMathOperator{\Sp}{Sp}

\DeclareMathOperator{\gln}{\GL_n\!\bb{C}}

\DeclareMathOperator{\SU}{SU}

\DeclareMathOperator{\rank}{rank}

\newcommand\G{\Gamma}
\newcommand\Gi{\G_{(i)}}
\newcommand\Gs{\G_{(s)}}
\newcommand\HGG{\Hom(\G,G)}
\newcommand\HGK{\Hom(\G,K)}
\newcommand\HGGmG{\HGG\!\sslash\! G}
\newcommand\one{\mathds{1}}
\newcommand\CC{\mathbb{C}}
\newcommand\ZZ{\mathbb{Z}}

\newtheorem*{namedtheorem}{\theoremname}
\newcommand{\theoremname}{testing}
\newenvironment{named}[1]{\renewcommand{\theoremname}{#1}\begin{namedtheorem}}{\end{namedtheorem}}

\theoremstyle{plain}
\newtheorem{thm}{Theorem}[section]
\newtheorem{lem}[thm]{Lemma}
\newtheorem{prop}[thm]{Proposition}
\newtheorem{cor}[thm]{Corollary}

\theoremstyle{definition}

\newtheorem*{rem}{Remark}

\title{A note on nilpotent representations}
\author{Maxime Bergeron}
\address{Department of Mathematics, The University of British Columbia, Room 121 - 1984 Mathematics Road, V6T 1Z2,  Vancouver BC, Canada}
\email{mbergeron@math.ubc.ca}
\urladdr{http://www.math.ubc.ca/~mbergeron}
\thanks{M.B.\ was supported by an NSERC Alexander Graham Bell CGS-D Scholarship}
\author{Lior Silberman}
\email{lior@math.ubc.ca}
\urladdr{http://www.math.ubc.ca/~lior/}
\thanks{L.S.\ was partly supported by an NSERC Discovery Grant}

\begin{document}

\begin{abstract}
Let $\Gamma$ be a finitely generated nilpotent group and let $G$ be a
complex reductive algebraic group.  The representation variety $\HGG$ and
the character variety $\HGGmG$ each carry a natural topology,
and we describe the topology of their connected components in terms of
representations factoring through quotients of $\Gamma$ by elements of its
lower central series.
\end{abstract}

\maketitle

\section{Introduction}
Let $G$ be the group of complex points of an affine algebraic group.
When $\Gamma$ is a finitely generated group, one may parametrize the
homomorphisms from $\Gamma$ to $G$ by the images of a finite generating set.
This realizes $\HGG$ as an (affine) algebraic set, carved out of
a finite product of copies of $G$ by the relations of $\Gamma$.
As a complex variety, $\HGG$ admits a natural Hausdorff topology
obtained from an embedding into affine space and it is easy to see (and
well-known) that the analytic space structure on $\HGG$ is independent of the
chosen presentation of $\Gamma$.  Here, we will only consider the case where $G$ is reductive though, in principle, the questions we address below can be asked
without this assumption.

These spaces of homomorphisms are of classical interest
(see Lubotzky--Magid  \cite{LubotzkyMagid:RepresentationVarieties} and the
references therein) and their algebraic topology has been the subject of
much recent scrutiny (see, for instance, \cite{AdemGomez:Structure,
AdemGomez:Classifying_preprint, AdemGomezLindTillman:Infinite_preprint,  
Baird:CohomSpaceTuples, CohenStafa:SpacesCommuting_preprint, CohenStafa:Survey,
GomezPettetSouto:Fundamental}), stemming in part from the  work of
\'Adem and Cohen \cite{AdemCohen:Commuting}.
In this context, it was recently shown by the first named author
\cite{Bergeron:TopNilRepns} that if $\Gamma$ is nilpotent and $K$ is a
maximal compact subgroup of $G$, then there is a strong deformation
retraction of $\HGG$ onto $\HGK$. This result was first established by
homotopy-theoretic methods for $\Gamma$ abelian by Pettet and Souto
\cite{PettetSouto:CommutingTuples} and for $\Gamma$ expanding nilpotent by Souto and
the second named author.  The result for arbitrary nilpotent groups was obtained
in \cite{Bergeron:TopNilRepns} by replacing these earlier approaches
with algebro-geometric methods.  Nevertheless, the machinery developed
by Pettet--Souto and its followups is very well posed to the study of
topological invariants.  Accordingly, the goal of this note is to combine
these topological and algebro-geometric tools to obtain topological information
about representation spaces of nilpotent groups.

From now on, fix a non-abelian finitely generated $s$-step nilpotent group
$\G$.  Recall that this means that the lower central series, defined inductively
by
$$
  \G_{(1)} = \G,\quad\quad\G_{(i+1)} = [\G,\Gi]
$$
has $\Gs$ non-trivial but $\G_{(s+1)} = \{e\}$.
The epimorphism $\G\raw\G/\Gi$ induces an embedding
$$
  \Hom(\G/\Gi,G)\raw\HGG
$$
which (for general groups $\G$ and $G$) is not even an open map.
Nevertheless, we will show:

\begin{thm}\label{thm:induction}
Let $\G$ be a finitely generated nilpotent group and let $G$ be the group of complex
points of a (possibly disconnected) reductive algebraic group.
For all $i\geq 2$, the inclusion 
$$
  \Hom(\G/\Gi,G)\xrightarrow{\iota} \HGG
$$
is a homotopy equivalence onto the union of those components of the target
intersecting the image of $\iota$.
\end{thm}
Consider $\HGG$ as a based space by taking the trivial representation as
the base point. In this case,   Theorem \ref{thm:induction} implies that the connected
components 
$$
  \HGG_\one \subset \HGG \text{ and }
  \Hom(\G/\Gi,G)_\one \subset \Hom(\G/\Gi,G)
$$
of the trivial representation are homotopy equivalent for all $i\geq 2$.
Using this, we will describe the homotopy type of the component of the trivial
representation in terms of abelian representations:
\begin{cor}\label{cor:abelianization}
For $\G$ and $G$ as in \emph{Theorem \ref{thm:induction}},
there is a homotopy equivalence
$$
  \HGG_\one \simeq \Hom(\ZZ^{\rank H_1(\G;\ZZ)},G)_\one.
$$
\end{cor}

To introduce the other space we study, note first that the action of $G$ on
itself by conjugation induces an action on $\HGG$ and conjugate
homomorphisms are often considered equivalent (this is the usual notion of
equivalence of representations in $\gln$).  Accordingly one often
wishes to understand the associated quotient but,   unfortunately, the naive
topological quotient is not a nice space: it need not even be Hausdorff.
In order to ``repair" this space, we use the affine geometric invariant theory
quotient $\HGG/\!\!/G$ instead.
This so-called \emph{character variety} is usually endowed with the structure
of an affine variety but, for our purposes, it may be constructed topologically
as the universal quotient in the category of Hausdorff spaces
(see Brion--Schwarz \cite{BrionSchwartz:Theorie}).  The systematic study of the
topology of theses spaces has seen much recent development (see, for instance,
\cite{BiswasFlorentino:Character_preprint,
 BiswasLawtonRamras:Fundamental_preprint, FlorentinoLawton:TopModSpace,
 FlorentinoLawton:TopCharVar,  LawtonRamras:Covering_preprint}).
Concentrating on the component of the trivial representation, we will use Corollary \ref{cor:abelianization} to prove:
\begin{cor}\label{cor:fundamental}
Let $\G$ be a finitely generated nilpotent group and let $G$ be the group of complex
points of a reductive algebraic group.  Then
\begin{enumerate}
\item $\pi_1\left(\HGG_\one\right) \cong \pi_1(G)^{\,\rank H_1(\G;\ZZ)},$ and
\item $\pi_1\left(\left(\HGGmG\right)_\one\right) \cong \pi_1(G/[G,G])^{\,\rank H_1(\G;\ZZ)}.$
\end{enumerate}
\end{cor}

\begin{cor}\label{cor:cohomology}
Let $G$ be the group of complex points of a connected reductive algebraic
group, let $T\subset G$ be a maximal algebraic torus and let $W$ be the Weyl
group of $G$.  If $\G$ is a finitely generated nilpotent group and $F$ is a
field of characteristic $0$ or relatively prime to the order of $W$, then:
\begin{enumerate}
\item $H^{\ast}\left(\HGG_\one;F\right)
              \cong H^{\ast}(G/T \times T^{\,\rank H_1(\G;\ZZ)};F)^W,$ and
\item $H^{\ast}\left(\left(\HGGmG\right)_\one;F\right)
              \cong H^\ast(T^{\,\rank H_1(\G;\ZZ)};F)^W.$
\end{enumerate}
\end{cor}

While the results above indicate many similarities between representation
spaces of abelian and non-abelian nilpotent groups, the latter have a much
richer topology than the former.  For instance, recall that for a connected
semisimple group $S$, the variety $\Hom(\ZZ^2,S)$ is irreducible and thus
connected \cite{Richardson:CommutingVarieties}. Moreover, $\Hom(\ZZ^r,\SL_n\CC)$,
$\Hom(\ZZ^r,\Sp_{2n}\CC)$ and the corresponding character varieties are
connected for all values of $r$ and $n$.  The situation for non-abelian
nilpotent groups is markedly different:
\begin{thm}\label{thm:disconnected}
Let $G$ be the group of complex points of a (possibly disconnected)
reductive algebraic group. If $\Gamma$ is a finitely generated nilpotent group
which surjects onto a finite non-abelian subgroup of $G$,
then $\HGG$ and $\,\HGGmG$ are both disconnected topological spaces.
\end{thm}
Since non-abelian free nilpotent groups and Heisenberg groups surject onto
the non-abelian nilpotent group of order $8$, this implies:
\begin{cor}\label{cor:heisenberg}
Let $\Gamma$ be a non-abelian free nilpotent group or a Heisenberg group.
If $G$ is the group of complex points of a reductive algebraic group, then
$\HGG$ and $\HGGmG$ are connected if and only if $G$ is an algebraic torus.
\end{cor}

\begin{rem}
All of the preceding statements remain true when $G$ is replaced by a compact
Lie group $K$.  In fact, we will prove most of them in this setting before
obtaining the complex reductive case via a homotopy equivalence.
\end{rem}

\textbf{Outline of the paper.}
We begin Section \ref{interesting bundle} by   describing compact  representation spaces using a fibre bundle.  Then, in Section \ref{proofs}, we use this bundle to prove Theorem \ref{thm:induction} and Theorem \ref{thm:disconnected} along with their various corollaries.

\textbf{Acknowledgements.}
The authors would like to thank Alejandro \'Adem and Alexandra Pettet for their comments on a preliminary version of this note. The first named author would also like to thank 
Man Chuen Cheng, Justin Martel and Juan Souto for many stimulating conversations.
 
\section{An interesting bundle}\label{interesting bundle}
The goal of this section is to prove the following key proposition:
\begin{prop}\label{bundle}
Let $K$ be a (possibly disconnected) compact Lie group. If $\Gamma$ is an $s$-step nilpotent group with $s\geq 2$, then the set of abelian groups
$$
{\cal{F}}:=\{\rho(\Gamma_{(s)})\subset K: \rho\in\Hom(\Gamma,K)\}
$$
admits a homogeneous manifold structure with finitely many connected components 
for which the projection map
\begin{equation}
p:\Hom(\Gamma,K)\raw {\mathcal{F}}\text{, }p(\rho)=\rho(\Gamma_{(s)})
\end{equation}
 is a locally trivial fibre  bundle.
\end{prop}

The proof of Proposition \ref{bundle} relies on  the following lemma:
\begin{lem}\label{Obound}
For all $m\in\bb{N}$ there is an $O=O(m)\in\bb{N}$ such that, if $N\subset \SU_m$ is an $s$-step  nilpotent group with $s\geq 2$, then $N_{(s)}$ is an abelian subgroup of $\SU_m$ of order bounded by $O$.
\end{lem}
\begin{proof}
Recall that $N_{(s)}$ is an abelian subgroup of $\SU_m$ contained in the centre of $N$. As such, there is a direct sum decomposition $\bb{C}^m=V_1\oplus\ldots\oplus V_r$ and $r$ characters $\chi_1,\ldots,\chi_r:N_{(s)}\raw\bb{C}^\times$ such that
 $\chi_i\neq\chi_j$ for all $i\neq j$ and
  $\gamma(v)=\chi_i(\gamma)\cdot v$ for all $\gamma\in N_{(s)}$ and $v\in V_i$.
  Moreover, for all $g\in N$, $\gamma\in N_{(s)}$ and $v\in V_i,$ we have 
$$\gamma(g(v))=g(\gamma(v))=g(\chi_i(\gamma)\cdot v)=\chi_i(\gamma)\cdot g(v).
$$ This allows us to consider the restrictions of the determinant homomorphism 
$$
\mathrm{det}_i:N\raw \bb{C}^\times,\,\,\,\mathrm{det}_i(g):=\det(g|_{V_i})
$$
where, since $\bb{C}^\times$ is abelian and  $s\geq 2$, the subgroup $N_{(s)}$ must be contained in $\ker(\mathrm{det}_i)$. This means that for all $i$ and all $\gamma\in N_{(s)},$ we have 
$$
\mathrm{det}_i(\gamma)=\chi_i(\gamma)^{\dim V_i}=1,
$$ so $\chi_i(\gamma)$ is always a root of unity of order bounded by $m$. Consequently, $N_{(s)}$ is conjugate in $\SU_m$ to a subset of those diagonal matrices whose diagonal elements are roots of unity of order bounded by $m$. This completes the proof since the order of this finite set does not depend on $s$.
\end{proof}
\begin{proof}[Proof of Proposition \ref{bundle}]
Choose a faithful embedding of $K$ into $\SU_m$.
By Lemma \ref{Obound}, there is a constant $O\in\bb{N}$ uniformly bounding the order of abelian subgroups of $K$ occurring as the image  of $\Gamma_{(s)}$ under homomorphisms $\rho:\Gamma\raw K$. In order to give 
${\mathcal{F}}$ a homogeneous manifold structure, we first consider the slightly larger set 
$$
\tilde{\cal{F}}:=\{A\subset K:A\text{ is an abelian subgroup of order bounded by }O\}.
$$
Observe that $K^o$ (the identity component of $K$) acts by conjugation on $\tilde{\cal{F}}$ with closed stabilizers. As such, we can endow $\tilde{\mathcal{F}}$ with the orbifold structure with respect to which each $K^o$-orbit is a connected homogeneous $K^o$-manifold (see \cite{OnishchikVinberg:LieGroups}). Concretely, if we define the ``connected normalizer'' as $N_{K^o}(H):=N_{K}(H)\cap K^o,$
then 
 the connected component of $H\in\tilde{\mathcal{F}}$ is identified with $K^o/N_{K^o}(H)$. 
 Having a topology on each $K^o$-orbit, we endow $\tilde{\cal{F}}$ with the disjoint union topology. Since $K$ is a compact Lie group, there are only finitely many conjugacy classes of abelian subgroups of $K$ of order bounded by $O$ and, 
in particular, $\tilde{\mathcal{F}}$ has only finitely many connected components.

A  homomorphism $\rho:\Gamma_{(s)}\raw K$ need not extend to the full group
$\Gamma$ so the map 
$$p:\Hom(\Gamma,K)\raw \tilde{\mathcal{F}}\text{, }p(\rho)=\rho(\Gamma_{(s)})$$ may not be surjective.
Accordingly, we denote $\mathcal{F}:=p(\Hom(\Gamma,K))$ and observe by $K^o$-equivariance of $p$ that it is a union of connected components of $\tilde{\mathcal{F}}$.
   Let $\mathcal{Z}\subset \mathcal{F}$ denote the connected component of a
finite abelian subgroup $H\in \mathcal{F}$ and let
$\mathcal{H}:=p\inv(\mathcal{Z})\subset \Hom(\Gamma,K)$. Since
$\cal{F}$ has only finitely many components, it follows that $p$ is a continuous map and it now suffices to show that $p\colon\mathcal{H}\raw \mathcal{Z}$ is a locally trivial fibre bundle. Observing once again that $p$ is $K^o$-\,equivariant, this follows at once from \cite[Proposition 2.3.2]{Bredon:IntroTranfGps}. More concretely, letting $\mathcal{H}(H):=p\inv(H)$, we can identify the restriction of $p$ to $\mathcal{H}$ with the twisted product $$(K^o\times\mathcal{H}(H))/N_{K^o}(H)\raw K^o/N_{K^o}(H)$$ where $N_{K^o}(H)$ acts on $K^o$ (resp. $\mathcal{H}(H)$) by right multiplication (resp. conjugation). \end{proof}

\section{Proofs of the main results}\label{proofs}
Let $G$ be the group of complex points of a (possibly disconnected) reductive algebraic group and recall that such a $G$ necessarily arises as the complexification of a (possibly disconnected) compact Lie group $K$. In this section, we use Proposition \ref{bundle} to prove the results mentioned in the introduction. In most cases, we prove a corresponding statement with $K$ in lieu of $G$ before obtaining the claimed result. We refer the reader to Onishchick--Vinberg \cite{OnishchikVinberg:LieGroups} for basic facts about Lie groups and complex algebraic groups.

Let $\Gamma$ be an $s$-step nilpotent group with $s\geq 2$ and recall that, for all $i,$ the epimorphism $\Gamma\raw \Gamma/\Gamma_{(i)}$ induces an embedding $\Hom(\Gamma/\Gamma_{(i)},K)\raw \Hom(\Gamma,K).$ Often, we shall abuse notation and identify $\Hom(\Gamma/\Gamma_{(i)},K)$ with its image under this embedding. As a first consequence of Proposition \ref{bundle} we obtain:
\begin{prop}\label{compactinduction}
Let $K$ be a (possibly disconnected) compact Lie group. If $\Gamma$ is a finitely generated nilpotent group then, for all $i\geq2$, the inclusion
$$
\Hom(\Gamma/\Gamma_{(i)},K)\xrightarrow{\iota} \Hom(\Gamma,K)
$$
is a homeomorphism onto the union of those components of the target intersecting the image of $\iota$.
\end{prop}
\begin{proof}
We proceed by induction on the nilpotence step of $\Gamma$. Recall from Proposition \ref{bundle} that 
$$
p:\Hom(\Gamma,K)\raw {\mathcal{F}}\text{, }p(\rho)=\rho(\Gamma_{(s)})
$$ is a locally trivial bundle.
If
$\Gamma$ is $2$-step nilpotent, then the image of $\iota$ consists of all
representations factoring through the abelianization of $\Gamma$, that is those
such that $\rho(\Gamma_{(2)}) = \{ e_K \}$.  Since $e_K$ is fixed by the
conjugation action of $K$, the subgroup $\{e_K\}\in\mathcal{F}$ is an isolated
point in the given topology.  Thus, for any $\rho\in p\inv(e_K)$,
the full connected component of $\rho$ (which is path-connected) has trivial
restriction to $\Gamma_{(2)}$ and we see that $p\inv\left(\{e_K\}\right)$
is the union of the connected components it intersects, 
completing the proof in this case.

Suppose now that $\Gamma$ is $s$-step nilpotent. If $i=s$, the same argument
as for the base case applies.  Otherwise, $i<s$ and then 
$$\Gamma/\Gamma_{(i)}\cong (\Gamma/\Gamma_{(s)})/(\Gamma_{(i)}/\Gamma_{(s)})$$
where the nilpotence step of $(\Gamma/\Gamma_{(s)})$ is $s-1$. 
As such, 
the induction hypothesis implies that each of the following two embeddings  
 $$
 \Hom(\Gamma/\Gamma_{(i)},K)\raw \Hom(\Gamma/\Gamma_{(s)},K)\raw \Hom(\Gamma,K)
 $$
  is a homeomorphisms onto those components of the target intersecting its image and, consequently, that the same holds for their composition.
  \end{proof}
  \begin{proof}[Proof of Theorem \ref{thm:induction}]
The theorem follows at once by  \cite[Theorem I]{Bergeron:TopNilRepns}.   \end{proof}
We can now prove:
\begin{cor}\label{compactabelianization}
If $\Gamma$ and $K$ are as in \emph{Proposition} $\ref{compactinduction}$, then there is a homeomorphism 
$$
\Hom(\Gamma,K)_{\mathds{1}}\cong\Hom(\bb{Z}^{\,rank\,H_1(\Gamma;\bb{Z})},K)_\mathds{1}.
$$
\end{cor}
\begin{proof}
By Proposition \ref{compactinduction}, we have a homeomorphism 
$$
\Hom(\Gamma,K)_\mathds{1}\cong\Hom(H_1(\Gamma;\bb{Z}),K)_\mathds{1}.
$$
 Since  $H_1(\Gamma;\bb{Z})=\Gamma/[\Gamma,\Gamma]$ is a finitely generated abelian group, we may identify $H_1(\Gamma;\bb{Z})$ with $\bb{Z}^r\oplus A$ where $r:={\,\mathrm{rank} \, H_1(\Gamma;\bb{Z})}$ and  $A$ is a finite abelian group. At this point we would like to show that $\Hom(\bb{Z}^r\oplus A,K)_\mathds{1}=\Hom(\bb{Z}^r,K)_\mathds{1}$.
Seeking a contradiction, suppose that $\rho_0\in \Hom(\bb{Z}^r\oplus A,K)_\mathds{1}$ maps $A$ non-trivially into $K$. By assumption, there is a continuous path of representations $[0,1]\mapsto \rho_t$ starting at $\rho_0$ and ending at the trivial representation $\rho_1=\mathds{1}$. 
But now, this path induces a continuous deformation in $\Hom(A,K)$ of the representation $\rho_0|_A$ to the trivial representation. This is impossible since  Lie groups contain no small subgroups.
\end{proof}
\begin{proof}[Proof of Corollary \ref{cor:abelianization}]
The corollary follows at once by \cite[Theorem I]{Bergeron:TopNilRepns}.
\end{proof}
Using this, we immediately  obtain:
\begin{named}{Corollary \ref{cor:fundamental}}
Let $G$ be the group of complex points of a  reductive algebraic group. If $\Gamma$ is a finitely generated nilpotent group, then: 
\begin{enumerate}
\item $\pi_1(\Hom(\Gamma,G)_\mathds{1})\cong \pi_1(G)^{\,rank\, H_1(\Gamma;\bb{Z})},$ and
\item $\pi_1((\Hom(\Gamma,G)/\!\!/G)_\mathds{1})\cong \pi_1(G/[G,G])^{\,rank\,{H_1(\Gamma;\bb{Z})}}.$
\end{enumerate}
\end{named}
\begin{proof}
The two formulas follow at once from Corollary \ref{cor:abelianization} by the main results of  
G\'omez--Pettet--Souto \cite{GomezPettetSouto:Fundamental} and Biswas--Lawton--Ramras  \cite{BiswasLawtonRamras:Fundamental_preprint}.\end{proof}


In order to prove our second corollary, we need the following:
 \begin{lem}\label{connectedquotient}
 If $K$ is a compact Lie group and $\Gamma$ is a finitely generated nilpotent group, then 
 $\Hom(\Gamma,K)_\mathds{1}/K=(\Hom(\Gamma,K)/K)_\mathds{1}.$ In particular, $\Hom(\Gamma,K)$ is connected if and only if $\Hom(\Gamma,K)/K$ is connected.
 \end{lem}
\begin{proof}
Recall from Corollary \ref{compactabelianization} that any $\rho\in\Hom(\Gamma,K)_\mathds{1}$ factors through the torsion free part of $H_1(\Gamma;\bb{Z})$. As such, by \cite[Lemma 4.2]{Baird:CohomSpaceTuples}, $\rho\in\Hom(\Gamma,K)_\mathds{1}$ if and only if there is a torus $T\subset K$ such that $\rho(\Gamma)\subset T$. 
 Since this property is preserved under conjugation by elements of $K$, it follows that $(\Hom(\Gamma,K)/K)_\mathds{1}$ coincides with the quotient $\Hom(\Gamma,K)_{\mathds{1}}/K$.
\end{proof}
We can now prove the cohomological formulas mentioned in the introduction.
\begin{named}{Corollary \ref{cor:cohomology}}
   Let $G$ be the group of complex points of a  connected reductive algebraic group, let  $T\subset G$ be a maximal algebraic torus and let $W$ be the Weyl group of $G$. If $\Gamma$ is a finitely generated nilpotent group and $F$ is a field of characteristic $0$ or relatively prime to the order of $W$, then:
   \begin{enumerate}
\item $H^{\ast}(\Hom(\Gamma,G)_\mathds{1};F)\cong H^{\ast}(G/T\times T^{\,rank\, H_1(\Gamma;\bb{Z})};F)^W,$ and
\item $H^\ast((\Hom(\Gamma,G)/\!\!/G)_\mathds{1};F)\cong H^\ast(T^{\,rank\, H_1(\Gamma;\bb{Z})};F)^W.$
\end{enumerate}
\end{named}

\begin{proof}[Proof of Corollary \ref{cor:cohomology}]
Following Pettet--Souto \cite[Corollary 1.5]{PettetSouto:CommutingTuples}, 
let $K\subset G$ be a maximal compact subgroup such that $T_K:=T\cap K$ is a maximal torus in $K$. Notice  that, for any $r\in\bb{N}$, $$
K/T_K\times T^{\,r}\raw G/T \times T^{\,r}
$$
is a $W$-equivariant homotopy equivalence and, in particular, that
\begin{equation}\label{baird}
H^\ast(K/T_K\times T^{\,r})^W\cong H^\ast(G/T \times T^{\,r})^W.
\end{equation}
Here, it follows from Baird \cite[Theorem 4.3]{Baird:CohomSpaceTuples} that the left hand side of the equation is isomorphic to $H^\ast(\Hom(\bb{Z}^r,K)_\mathds{1})$. 
Now, letting $r:={\mathrm{rank}\,H_1(\Gamma;\bb{Z})}$, our first formula follows at once from the homotopy equivalences 
$$
\Hom(\Gamma,G)_{\mathds{1}}\simeq\Hom(\bb{Z}^r,G)_\mathds{1}\simeq \Hom(\bb{Z}^r,K)_\mathds{1}
$$ provided by Corollary \ref{cor:abelianization} and \cite[Theorem I]{Bergeron:TopNilRepns}. Finally, it is also due to Baird \cite[Remark 4]{Baird:CohomSpaceTuples} that $ \Hom(\bb{Z}^r,K)_\mathds{1}/K\cong T_K^{\,r}/W$ so
our second formula follows from the homotopy equivalence and homeomorphisms
$$
(\Hom(\Gamma,G)/\!\!/G)_\mathds{1}\simeq (\Hom(\Gamma,K)/K)_\mathds{1}\cong \Hom(\Gamma,K)_\mathds{1}/K\cong \Hom(\bb{Z}^r,K)_\mathds{1}/K.
$$
provided by  \cite[Theorem II]{Bergeron:TopNilRepns}, Lemma \ref{connectedquotient} and Corollary \ref{compactabelianization}.
\end{proof}
\begin{rem}
The homotopy types of distinct components of representation spaces are typically different.  For instance, if we take $\Gamma$ to be the discrete Heisenberg group $H_3(\bb{Z})$, then $\Hom(\Gamma,\SL_2\bb{C})$ decomposes into a simply-connected component  and a non simply-connected component. In fact, this phenomenon already occurs for $\Gamma$ abelian as illustrated in  G\'omez--Adem \cite{AdemGomez:Structure} and G\'omez--Pettet--Souto \cite{GomezPettetSouto:Fundamental}.
\end{rem}

\begin{named}{Theorem \ref{thm:disconnected}}
Let $G$ be the group of complex points of a reductive algebraic group. If $\Gamma$ is a finitely generated nilpotent group which surjects onto a finite non-abelian subgroup of $G$, then $\Hom(\Gamma,G)$ and $\Hom(\Gamma,G)/\!\!/G$ are both disconnected.
\end{named}
\begin{proof}
Let $\psi:\Gamma\raw N$ be a surjective homomorphism onto a finite non-abelian subgroup of $G$ and let $K$ be a maximal compact subgroup of $G$ containing $N$. Notice in particular that $\psi\in\Hom(\Gamma,K)\subset \Hom(\Gamma,G)$. Since $\Hom(\Gamma,K)\simeq \Hom(\Gamma,G)$ and $\Hom(\Gamma,K)/K\simeq \Hom(\Gamma,G)/\!\!/G$ by \cite{Bergeron:TopNilRepns}, and since $\Hom(\Gamma,K)$ is disconnected if and only if $\Hom(\Gamma,K)/K$ is disconnected by  Lemma \ref{connectedquotient}, it suffices to prove that $\Hom(\Gamma,K)$ is disconnected.

 Seeking a contradiction, suppose that $\Hom(\Gamma,K)$ is connected and recall from  Proposition \ref{compactinduction} that, in this case, $\Hom(\Gamma/\Gamma_{(i)},K)$ is connected for all $i\geq 2$.  
Choose a minimal $s\in\bb{N}$ with the property that $\psi(\Gamma_{(s+1)})=e_K$ and denote the $s$-step nilpotent group $\Gamma/\Gamma_{(s+1)}$ by $\hat\Gamma$. If we consider the fibre bundle (c.f. Proposition \ref{bundle}) $$
p:\Hom(\hat\Gamma,K)\raw \mathcal{F}\text{, }p(\rho)=\rho(\hat\Gamma_{(s)}),
$$ then $p(\psi)=\psi(\hat\Gamma_{(s)})\neq e_K$. As such, by Proposition \ref{compactinduction} and our assumptions, 
$$
\psi\notin\Hom(\hat\Gamma,K)_{\mathds{1}}\cong\Hom(\Gamma/\Gamma_{(s+1)},K)_{\mathds{1}}\cong\Hom(\Gamma,K)_{\mathds{1}}\cong\Hom(\Gamma,K)
$$ and this contradiction completes the proof.
\end{proof}

\begin{named}{Corollary \ref{cor:heisenberg}}
Let $\Gamma$ be a non-abelian free nilpotent group or a Heisenberg group. If $G$ is the group of complex points of a reductive algebraic group, then $\Hom(\Gamma,G)$ and $\Hom(\Gamma,G)/\!\!/G$ are connected if and only if $G$ is an algebraic torus.
\end{named}
 \begin{proof}
 If $G$ is disconnected or not simply-connected then \cite[Corollary 1.3]{PettetSouto:CommutingTuples}, \cite[Theorems I and II]{Bergeron:TopNilRepns} and Lemma \ref{connectedquotient}  show that $\Hom(H_1(\Gamma;\bb{Z}),G)$ and $\Hom(H_1(\Gamma;\bb{Z}),G)/\!\!/G$ are disconnected. As such, it suffices to consider the case where $G$ is simply-connected. 
Notice that such a $G$ contains a subgroup isomorphic to $\SL_2\bb{C}$ and, since $\SL_2\bb{C}$ contains a copy of the non-abelian group $Q$ of order $8$ generated by the matrices
$$\left(\begin{array}{cc}i & 0 \\0 & -i\end{array}\right)\text{ and }\left(\begin{array}{cc}0 & 1 \\-1 & 0\end{array}\right),$$
 so does $G$.
 Since $Q$ is a $\bb{Z}/2\bb{Z}$ central extension of $\bb{Z}/2\bb{Z}\times \bb{Z}/2\bb{Z}$, it follows that if $\Gamma$ is either a non-abelian free nilpotent group or a Heisenberg group, then $\Gamma$ surjects onto $Q$. The claim now follows from Theorem \ref{thm:disconnected}.
 \end{proof}

\bibliographystyle{plain}
\bibliography{nilpotentnote}
\end{document}